\documentclass[12pt]{amsart}
\usepackage{amsmath}
\usepackage{amsxtra}
\usepackage{amscd}
\usepackage{amsthm}
\usepackage{amsfonts}
\usepackage{amssymb}
\usepackage {pstricks}
\usepackage{pstricks,pst-node}
\usepackage[all]{xy}

\usepackage{mathrsfs}

\usepackage{verbatim}

\newtheorem{theorem}{Theorem}[section]
\newtheorem*{theorem*}{Theorem}
\newtheorem{lemma}[theorem]{Lemma}

\newtheorem*{proposition*}{Proposition}
\newtheorem{corollary}[theorem]{Corollary}

\theoremstyle{definition}
\newtheorem{definition}[theorem]{Definition}
\newtheorem{remark}[theorem]{Remark}
\newtheorem*{conjecture*}{Conjecture}
\newtheorem*{notation*}{Notation}

\newtheorem{example}[theorem]{Example}

\numberwithin{equation}{section}

\def\1{1\kern-.3em1}

\newcommand{\Z}{{\mathbb Z}}

\newcommand{\C}{{\mathbb C}}

\newcommand{\U}{{\rm U}}

\newcommand{\gl}{{\mathfrak{gl}}}

\newcommand{\fg}{{\mathfrak g}}

\newcommand{\fu}{{\mathfrak u}}

\newcommand{\fl}{{\mathfrak l}}
\newcommand{\fp}{{\mathfrak p}}

\newcommand{\id}{{\rm{id}}}
\newcommand{\Hom}{{\rm{Hom}}}
\newcommand{\End}{{\rm{End}}}

\newcommand{\CF}{{\mathcal F}}

\newcommand{\CW}{{\mathscr W}}

\newcommand{\Cl}{{\mathcal{C}l}}

\newcommand{\Der}{{\rm Der}}

\newcommand{\Sym}{{\rm{Sym}}}

\newcommand{\sdim}{{\rm{sdim}}}

\newcommand{\im}{{\rm{im\, }}}

\newcommand{\HC}{\widehat{C}}
\newcommand{\HH}{\widehat{H}}
\def\ad{{\rm ad}}

\keywords{Lie superalgebra cohomology, Weyl superalgebras, integral forms}

\subjclass[2010]{17B56, 18G35 (primary), 81R05 (secondary)}

\begin{document}
\title[Mixed cohomology of Lie superalgebras]
{Mixed cohomology of Lie superalgebras}
\author[Yucai Su]{Yucai Su}
\address{Department of Mathematic, Tongji University,  Shanghai, China}
\email{ycsu@tongji.edu.cn}

\author[R. B. Zhang]{R. B. Zhang}
\address{School of Mathematics and Statistics,
The University of Sydney, Sydney, NSW 2006,  Australia}
\email{ruibin.zhang@sydney.edu.au}

\begin{abstract}
We investigate a new cohomology of Lie superalgebras, which may be compared to
a de Rham cohomology of Lie supergroups involving both differential and integral forms.
It is defined by a BRST complex of Lie superalgebra modules,
which is formulated in terms of a Weyl superalgebra and
incorporates inequivalent representations of the bosonic Weyl subalgebra.
The new cohomology includes the standard Lie superalgebra cohomology as a special case.
Examples of new cohomology groups are computed.
\end{abstract}
\maketitle


\section{Introduction}\label{sect:intro}

The de Rham cohomology of a Lie group can be equivalently reformulated  in terms of the cohomology of its Lie algebra \cite{CE} (also see \cite[\S 7]{We} for Lie algebra cohomology).  This fact had profound impact on the development of Lie theory.
One may try to cast a de Rham theory of a Lie supergroup into a similar algebraic setting by using the cohomology of its Lie superalgebra. However, this runs into the difficulty that the Lie superalgebra cohomology   \cite{F, FL, SchZ} studied so far in the literature is not adequate for this purpose. The problem is rooted in supergeometry, thus is not present in the Lie group context.

To see the cause of the problem, we note the following fact about supermanifolds
(see, e.g., \cite{DM, V} for introductions), which is not widely known.
Beside differential forms, there also exist integral forms and mixtures of the two types of forms \cite{BL, W} on supermanifolds. They are all necessary for defining integration, and in particular, for establishing a Stokes' theorem \cite{BL}.
A generalisation of de Rham theory to supermanifolds should take into account  differential-integral forms to capture new features of supergeometry.
However,  the cohomology of Lie superalgebras \cite{F, FL, SchZ} in the literature is a direct generalisation of the cohomology of Lie algebras \cite{CE}\cite[\S 7]{We}.
It is usually defined by a generalised Chevalley-Eilenberg complex \cite{SchZ} (also see Section \ref{sect:cohomology}), which corresponds to a complex of differential forms (tensored with a coefficient module). The integral forms are entirely discarded.

Our aim is to develop a Lie superalgebra cohomology which will take into full account of differential-integral forms.  We will call this a {\em mixed cohomology of Lie superalgebras}.

We now describe in more detail the background of the current work, and also outline the main ideas and techniques involved.

\subsubsection*{Differential-integral forms on supermanifolds}

Let us recall the elementary treatment of differential-integral forms on supermanifolds by Witten \cite{W}.
Given a supermanifold $M$ with the tangent bundle $TM$,  denote by $\Pi{TM}$ the tangent bundle with the parity of the fibre space reversed.  Thus the odd directions of the fibre space of $\Pi{TM}$ come from the even coordinates of $M$ (in particular,  for an ordinary manifold, the fibre of $\Pi{TM}$ is purely odd), and the even directions from the Grassmannian variable coordinates. The forms on $M$ are functions on $\Pi{TM}$: differential forms are polynomial functions, and integral forms \cite{BL} are distributions supported at $0$ in the even directions of the fibre space of $\Pi{TM}$. Mixtures of the two types of (generalised) functions are differential-integral forms.

\subsubsection*{Modules for Weyl superalgebras}

A key observation in  \cite[\S 3.2]{W} is that differential-integral forms at a point of a supermanifold $M$ constitute a module over a Weyl superalgebra. This Weyl superalgebra is generated by the coordinates of the fibre of $\Pi{TM}$ and their derivatives, thus is equal to the tensor product of a Clifford algebra of even degree and a bosonic Weyl algebra.
While the even degree Clifford algebra has a unique irreducible module (the fermionic Fock space) up to canonical isomorphisms,  the Weyl algebra has many non-isomorphic simple modules.
Different modules correspond to different functions on $\Pi{TM}$.  To see this, we note that
the bosonic Fock space for the Weyl algebra is cyclically generated by a vacuum vector, which is annihilated by all the even derivatives. Vectors in it correspond to differential forms. However, one may consider a module for the Weyl algebra cyclically generated by a vector which is annihilated by some even variables.  Such a vector is a distribution in these variables supported at $0$ (see Example \ref{ex:1-dim}). Applying the corresponding derivatives to the vector produces more distributions. These correspond to integral forms  \cite[\S 3.2]{W}.

\subsubsection*{Differential-integral forms in the Lie superalgebra context}

Conceptually we may regard a Lie superalgebra $\fg$ as the Lie superalgebra of left invariant tangent vector fields on the underlying supermanifold of a Lie supergroup $G$. By \cite{W}, the differential-integral forms at a point of $G$ are (generalised) functions on the parity reversed superspace $\Pi\fg$ of $\fg$. They  can be described in terms of modules for the Weyl superalgebra $\CW_{\Pi\fg^*}$ over $\Pi\fg^*=(\Pi\fg)^*$, the dual superspace of $\Pi\fg$.
We now apply such differential-integral forms to develop a theory of mixed cohomology of Lie superalgebras.

\subsubsection*{The BRST formalism}

Useful techniques for doing this are available in the physics literature, which originated from a method for quantising gauge theories known as the BRST formalism (see, e.g., \cite{BBH}), first  introduced by Becchi, Rouet and Stora,  and by Tyutin. The BRST method has since developed into a vast theory, which has been applied to many other areas with remarkable success, for example, to string theory, symplectic geometry and semi-infinite cohomology of affine Kac-Moody algebras.
We shall adapt the BRST method to implement the ideas of \cite{W} discussed above to Lie superalgebras.
As we will see in Section \ref{sect:mixed-cohomology}, the correspondence between differential-integral forms and modules for Weyl superalgebras \cite[\S 3.2]{W} becomes even more natural within the BRST framework.

\subsubsection*{Mixed cohomology of Lie superalgebras}

We define a mixed cohomology of a Lie superalgebra $\fg$ by a BRST complex of $\fg$-modules in Theorem \ref{thm:main}, which is formulated in terms of modules for the Weyl superalgebra $\CW_{\Pi\fg^*}$ over $\Pi\fg^*$.
The Weyl superalgebra is the tensor product of a Clifford subalgebra and a bosonic Weyl subalgebra, where the latter is not equal to $\C$ if the Lie superalgebra $\fg$ has a non-trivial odd subspace.  In this case,
$\CW_{\Pi\fg^*}$ has non-isomorphic simple modules $\CF_{\Pi\fg^*}(D)$, which are characterised by $\fg$-submodules $D$ of  $\Pi\fg^*$ called {\em mixing sets}.
These non-isomorphic simple modules for $\CW_{\Pi\fg^*}$ account for the differential-integral forms in the sense of \cite[\S3.2]{W} (see Section \ref{sect:forms}).
The BRST complex defining the mixed cohomology of $\fg$ with coefficients in a $\fg$-module $V$ has the superspace of cochains $C_D(\fg, V)=\CF_{\Pi\fg^*}(D)\otimes V$.
The standard Lie superalgebra cohomology \cite{F, FL, SchZ} is recovered  in
Theorem \ref{thm:reform} from the mixed cohomology in the special case $D=0$,
where $\CF_{\Pi\fg^*}(0)$ is the standard Fock space.

\subsubsection*{Examples}
To illustrate how the general theory of mixed cohomology works,  we calculate as examples various mixed $\fu$-cohomology groups of the general  linear Lie superalgebra in Section \ref{sect:compute}.  Some of these cohomology groups are highly non-trivial. It is the BRST method which enables us to carry out such computations explicitly.

\medskip

Now some comments are in order.

While the BRST method is the most direct way to reach Definition \ref{def:mixed} of the mixed cohomology of Lie superalgebras,  it is still very useful to reformulate the definition using more conventional homological algebraic methods \cite{We}, e.g., in terms of derived functors in a way analogous to \cite{Z04}.

We also mention that the mixed cohomology of Lie superalgebras  is a very natural object to study; one could have discovered it much earlier by simply placing the standard Lie superalgebra cohomology \cite{F, FL, SchZ} in the BRST framework.

We work over  the field $\C$ of complex numbers throughout.

\section{Standard  cohomology of  Lie superalgebras}\label{sect:standard-cohomology}

We recall the definition of the cohomology of Lie superalgebras \cite{FL, F, SchZ} here.  As we shall see in Section \ref{sect:mixed-cohomology}, it can be recovered from the BRST formalism for mixed cohomology.

\subsection{Vector superspaces}\label{sect:superspaces}

A vector superspace $V$ is a $\Z_2$-graded vector space $V=V_0\oplus V_1$, where $V_0$ and $V_1$ are called the even and odd subspaces respectively. Here $\Z_2:=\Z/2\Z$ is considered as an additive group. The degree of a homogeneous element $v\in V_0\cup V_1$ will be called the parity of $v$, and denoted by $[v]$.  Let $V$ and $W$ be any two vector superspaces.
The set $\Hom_\C(V, W)$ of homomorphisms is a vector superspace with
$\Hom_\C(V, W)_i = \oplus_{j+k = i } \Hom(V_j, W_k)$ for all $i\in\Z_2$.
The tensor product $V\otimes_\C W$ is also a vector superspace with
$(V\otimes_\C W)_i  =  \oplus_{j+k = i } V_j\otimes_\C W_k$.
The category of vector superspaces is a tensor category equipped with a canonical symmetry
\[
\tau_{V, W}: V\otimes W\longrightarrow W\otimes V, \quad v\otimes w \mapsto (-1)^{[v][w]} w\otimes v.
\]

Various types of algebras in this category will be called superalgebras of the corresponding types, e.g., associative superalgebras, Hopf superalgebras, and Lie superalgebras (see \cite{K, Sch} for the theory of Lie superalgebras).
We will consider only  $\Z_2$-graded modules for any superalgebra.

The dual space of $V$ is $V^*:=\Hom_\C(V, \C)$, which  is also $\Z_2$-graded.
Note that $V^*\otimes W^*$ is embedded in $(V\otimes W)^*$ such that for all $f\in V^*$ and $g\in W^*$,
\[
(f\otimes g)(v, w) = (-1)^{[g][v]}f(v) g(W), \qquad \forall\, v\in V, \ w\in W.
\]
If one of the vector superspaces is finite dimensional, this is an isomorphism.

There exists the parity change functor $\Pi$ on the category of vector superspaces, which acts as identity on morphisms,  and reverses the parity on objects, that is, $(\Pi V)_0=V_1$ and $(\Pi V)_1=V_0$ for any vector superspace $V$.  Now we have a canonical odd isomorphism $\pi: V\longrightarrow \Pi{V}$ of  vector superspace, which sends any homogeneous element of $V$ to the same element in $\Pi V$ but with the opposite parity.

Denote by $\Sym_r$ the symmetric group of degree $r$,
and by $\C\Sym_r$ its group algebra.
The symmetry $\tau_{V, V}: V\otimes V\longrightarrow V\otimes V$ extends to a representation
$
\nu_r: \C\Sym_r\longrightarrow \End_\C(V^{\otimes r})
$
of $\C\Sym_r$ on the $r$-th tensor power of $V$ such that for all the simple reflections $s_i = (i, i+1)$ for $1\le i\le i-1$,
\begin{eqnarray}\label{eq:symm-rep}
\nu_r: s_i \mapsto \tau_i = \underbrace{\id_V\otimes\dots\otimes\id_V}_{i-1}\otimes\tau_{V, V} \otimes\underbrace{\id_V\otimes\dots\otimes\id_V}_{r-i-1}, \quad \forall\, i.
\end{eqnarray}
Write $\ell(\sigma)$ for the length of an element $\sigma\in\Sym_r$, and let
\[
S_r= \frac{1}{r!}\sum_{\sigma\in\Sym_r}\sigma, \qquad \Sigma_r= \frac{1}{r!}\sum_{\sigma\in\Sym_r} (-1)^{\ell(\sigma)}\sigma.
\]
Then the symmetric and skew-symmetric $r$-th power of $V$ are respectively given by
\begin{eqnarray}\label{eq:S-Wedge}
S^r(V)=S_r(V^{\otimes r}), \qquad  {\bigwedge}^r(V)=\Sigma_r(V^{\otimes r}).
\end{eqnarray}
We have the following $\Z$-graded superspaces.
\begin{eqnarray}\label{eq:sym-skew}
S(V)=\bigoplus_{r=0}^\infty S^r(V), \qquad  \bigwedge(V)=\bigoplus_{r=0}^\infty{\bigwedge}^r(V).
\end{eqnarray}

\subsection{Cohomology of Lie superalgebras}\label{sect:cohomology}

Let us briefly discuss the standard Lie superalgebra cohomology following \cite[\S II]{SchZ}.
We refer to \cite{K, Sch} for the theory of Lie superalgebras.

Let $\fg=\fg_0\oplus\fg_1$ be a Lie superalgebra with the  super Lie bracket ${[\ , \ ]}: \fg\times\fg \longrightarrow \fg$.
Given any $\Z_2$-graded left modules $V$ and $W$ for $\fg$ (we only consider $\Z_2$-graded modules), their tensor product $V\otimes W$ is again a $\fg$-module with the diagonal action. The canonical symmetry  $\tau_{V, W}: V\otimes W \longrightarrow W\otimes V$ is a $\fg$-homomorphism.  This in particular implies that $S^r(W)$ and ${\bigwedge}^r(W)$ are $\fg$-submodules of $W^{\otimes r}$.
The dual vector superspace $W^*$ of a $\fg$-module $W$ has a $\fg$-module structure
defined by
$
(A f)(v) = - (-1)^{[A][f]}f(Av)
$
for all $A\in\fg$, $f\in W^*$ and $v\in W$.

Fix a $\fg$-module $V$, we consider the family of $\fg$-modules
\[
C^p(\fg, V) := {\bigwedge}^p(\fg^*)\otimes V,  \quad p\in\Z_+.
\]
For convenience, we also take $C^q(\fg, V)=0$ for all $q<0$.   The $\Z_2$-gradings of  $S^p(\Pi\fg^*)$ and $V$ naturally give rise to a $\Z_2$-grading for $C^p(\fg, V)$.
Define the following contraction maps \cite[(II.15)]{SchZ}
\[
\begin{aligned}
i_p:\ \ & C^p(\fg, V)\otimes \fg \longrightarrow C^{p-1}(\fg, V), \quad g\otimes A\mapsto g_A, \\
&g_A(A_1, A_2, \dots, A_{p-1}) = g(A, A_1, A_2, \dots, A_{p-1}),  \quad \forall\, A_i\in \fg,
\end{aligned}
\]
where for $g=a\otimes v\in C^p(\fg, V)$ with $a\in{\bigwedge}^p(\fg^*)$ and $v\in V$,
\[
g(A_1, A_2, \dots, A_p) = (-1)^{[v]\sum\limits_{i=1}^p[A_i]} a(A_1, A_2, \dots, A_p) v.
\]
These maps are $\fg$-homomorphisms.
We further define maps \cite[(II.18)]{SchZ}
\begin{eqnarray}\label{eq:d}
d^p: C^p(\fg, V) \longrightarrow C^{p+1}(\fg, V), \quad p=0, 1,  2, \dots,
\end{eqnarray}
inductively by
$
(d^p g)_A = (-1)^{[A][g]} A.g - d^{p-1} g_A$  for all $A\in \fg.
$

There exists an explicit formula for these maps, which can be described as follows \cite[(II.21)]{SchZ}.
Define linear maps $d^p_1, d^p_2:  C^p(\fg, V) \longrightarrow C^{p+1}(\fg, V)$ such that for any $g\in C^p(\fg, V)$ and for all $A_0, A_1, \dots, A_p\in\fg$,
\begin{align}
&(d^p_1 g)(A_0, A_1, \dots, A_p) = \sum\limits_{r=0}^p (-1)^{r+\Sigma_{r-1, r}}  A_r g(A_0, \dots, \hat{A}_r, \dots, A_p),  \label{eq:d-1}\\
&(d^p_2 g)(A_0, A_1, \dots, A_p)   \label{eq:d-2} \\
&\phantom{(d^p_2 g)}= \sum\limits_{r<s} (-1)^{s+ \Sigma_{r, s} - \Sigma_{s-1, s}}
g(A_0, \dots, A_{r-1}, [A_r, A_s], A_{r+1}, \dots, \hat{A}_s, \dots, A_p), \nonumber
\end{align}
with $\Sigma_{k, \ell}=[A_\ell] \big([g]+ \sum_{i=0}^k [A_i]\big)$ for $k<\ell$.
Then one can readily show that
\begin{eqnarray}
d^p = d^p_1+d^p_2, \quad \forall\, p.
\end{eqnarray}
The maps $d^p$ are $\fg$-homomorphisms which are even with respect to the $\Z_2$-grading. Furthermore, they satisfy $d^{p+1} d^p =0$ for all $p$.

\begin{definition}\cite{FL, SchZ} \label{def:standard}
The cohomology of  a Lie superalgebra $\fg$ with coefficients in a $\fg$-module $V$ is the homology of the differential complex
\[
(C(\fg, V), d): \qquad C^0(\fg, V) \stackrel{d^0}\longrightarrow C^1(\fg, V) \stackrel{d^1}\longrightarrow C^2(\fg, V)
\stackrel{d^2}\longrightarrow \cdots.
\]
Denote the  cohomology groups of $(C(\fg, V), d)$ by
$H^p(\fg, V) = \frac{\ker d^p}{\im d^{p-1}}$ for $p\in\Z_+$.
\end{definition}
To distinguish it from the mixed cohomology to be defined later,
we call this the standard cohomology of Lie superalgebras.
It has been much studied (see, e.g., \cite{CZ, F, GS, PS, S, SuZ}), and has wide applications,
for example, in the study of deformations of universal enveloping superalgebras such as quantum supergroups \cite{BGZ, Z02, Z04},  extensions of Lie superalgebras and modules \cite{SchZ, SuZ}, and support varieties of Lie supergroups \cite{BKN, LNZ}.  Another important application is in the solution of the character problem \cite{B, GS, PS, S, SZ} of classical Lie superalgebras, which relied on results on the super analogue of Kostant's $\fu$-cohomology \cite{CZ, PS, S} for specifically chosen parabolic sub superalgebras.  This super $\fu$-cohomology is well known to be deeply rooted in the geometry of  homogeneous superspaces \cite{PS, Z02}.

\begin{remark}
The complex in Defintion \ref{def:standard} reduces to the usual Chevalley-Eilenberg complex \cite[\S7.7]{We} if $\fg$ is an ordinary Lie algebra.
\end{remark}

\section{Mixed cohomology of Lie superalgebras}\label{sect:mixed-cohomology}

We present the BRST formulation of the mixed cohomology of Lie superalgebras in this section. As we shall see in Section \ref{sect:BRST-old}, the standard cohomology of Lie superalgebras discussed in Section \ref{sect:standard-cohomology} is a special case of the mixed cohomology.

\subsection{Weyl superalgebras}

Given a vector superspace $V$, the $\Sym_r$-representation
on $V^{\otimes r}$ defined by \eqref{eq:symm-rep}  is semi-simple. Thus $S^r(V)$ is a direct summand of $V^{\otimes r}$ as $\Sym_r$-module, and we may also define $S^r(V)$ as the quotient of $V^{\otimes r}$ by the complement. The skew symmetric power ${\bigwedge}^r(V)$ can be similarly defined as a quotient.  This gives an alternative description of $S(V)$ and $\bigwedge(V)$, which has the advantage of making their superalgebraic structures transparent.

Denote by $T(V)$ the tensor algebra of $V$. Let $J_S$ be the two-side ideal  generated by $v\otimes v' - \tau(v\otimes v')$ for all $v, v'\in V$, and let $J_\wedge$ be the two-side ideal  generated by $v\otimes v' + \tau(v\otimes v')$ for all $v, v'\in V$.  Then
\[
S(V)= T(V)/J_S, \qquad \bigwedge(V)=T(V)/J_\wedge.
\]
Hence both $S(V)$ and $\bigwedge(V)$  are associative algebras.

The tensor algebra $T(V)$ has a natural $\Z_+$-grading with $V$ being the degree $1$ subspace.  As $J_S$ and $J_\wedge$ are homogeneous, the algebras $S(V)$ and $\bigwedge(V)$  are both $\Z_+$-graded with the homogeneous subspaces given in \eqref{eq:sym-skew}.
The $\Z_2$-grading of $V$ induces a  $\Z_2$-grading for $T(V)$. Since both $J_S$ and $J_\wedge$ are $\Z_2$-graded ideals, $S(V)$ and $\bigwedge(V)$ are naturally $\Z_2$-graded.

Let $V^*$ be the dual vector superspace of $V$. We consider the  tensor algebra $T(V\oplus V^*)$ over $V\oplus V^*$. Let $J$ be the two-side ideal generated by the following elements
\[
v\otimes w - \tau_{V, V}(v\otimes w), \quad
{\bar v}\otimes {\bar w} - \tau_{V^*, V^*}({\bar v}\otimes \bar{w}), \quad
{\bar v}\otimes w - \tau_{V^*, V}({\bar v}\otimes w) - {\bar v}(w),
\]
for $v, w\in V$ and ${\bar v},  {\bar w}  \in V^*$.  Note that this ideal is $\Z_2$-graded.
%
The Weyl superalgebra of $V$ is the associative superalgebra
\[
\CW_V:= T(V\oplus V^*)/J.
\]
It is $\Z$-graded with $\deg(V)=1$ and $\deg(V^*)=-1$.
Note that  $\CW_V \cong  S(V)\otimes S(V^*)$ as vector superspaces.

For $V=V_0\oplus V_1$, we may regard $V_0$ (resp. $V_1$) as a purely even (resp. odd) vector superspace, and consider $\CW_{V_0}$ and $\CW_{V_1}$. Then $\CW_{V_0}$ is the usual Weyl algebra over $V_0$, and $\CW_{V_1}$ is the Clifford algebra generated by $V_1\oplus V^*_1$, thus is of even degree. We have
 $\CW_V=\CW_{V_0}\otimes\CW_{V_1}$ as superalgebra.

If $d_1:=\dim V_1<\infty$, the Clifford algebra $\CW_{V_1}$ has a unique simple module,  the fermionic Fock space, which is $2^{d_1}$-dimensional and is given by
\begin{eqnarray*}
\CF_{V_1}(0)=\frac{\CW_{V_1}}{\sum\limits_{\partial \in {V_1}^*}\CW_{V_1}\partial} = S(V_1),
\end{eqnarray*}
where $S(V_1)$ is defined by \eqref{eq:S-Wedge} with $V_1$ regarded as a purely odd vector superspace.
We may also construct a simple module, say, by taking $\CF_{V_1}(V_1):=\frac{\CW_{V_1}}{\CW_{V_1}V_1} = S(V_1^*)$, but this does not lead to anything new, since $\CF_{V_1}(V_1)\cong\CF_{V_1}(0)$ with the isomorphism determined by a vector space isomorphism  $S^0(V_1^*)\cong S^{d_1}(V_1)$ which is unique up to scalar multiples.

On the other hand, there are many non-isomorphic irreducible representations of the Weyl algebra $\CW_{V_0}$ even when $V_0$ is finite dimensional.  Given any subspace $D_0\subset V_0$, we let $D_0^\perp =\{{\bar v}\in V_0^*\mid {\bar v}(w)=0, \ \forall\, w\in D_0\}$.  Then we have the corresponding simple $\CW_{V_0}$-module
\[
\CF_{V_0}(D_0)=\frac{\CW_{V_0}}{\sum\limits_{\partial\in D_0^\perp} \CW_{V_0} \partial+ \sum\limits_{v\in D_0} \CW_{V_0} v}.
\]
In the extreme case with $D_0=0$, the module $\CF_{V_0}(0)$ is the standard Fock space of the Weyl algebra, where the vacuum vector is $1$.

\begin{example} \label{ex:1-dim}
Consider the case $V=\C x$ with $x$ being even. Then $\CW_V$ as an associative algebra is generated by $x$ and $\partial=\frac{d}{d x}$ subject to the relation $\partial x - x \partial =1$.  We have
\[
\CF_V(0)=\C[x], \quad \CF_V(V)=\C[\partial].
\]

We can interpret $\C[\partial]$ as a space of distributions in $x$ supported at $0$. To do this, we will regard $x$ as a real variable, and consider complex valued functions $f(x)$ in $x$.
Let $\delta(x)$ be Dirac's delta-function, which satisfies
\begin{eqnarray}\label{eq:derivative-delta}
\int\limits_{-\infty}\limits^\infty f(x) \frac{d^k\delta(x)}{d x^k}dx = (-1)^k \frac{d^k f(0)}{d x^k}, \quad k=0, 1, \dots,
\end{eqnarray}
for any function $f(x)$ such that its $k$-order derivative $\frac{d^k f(0)}{d x^k}$ exists at $0$.

Let $\mathcal{D}=\sum_{k\ge 0} \C\frac{d^k\delta(x)}{d x^k}$, and endow it with a $\Z$-grading such that $\frac{d^k\delta(x)}{d x^k}$ is at degree $-k$.  Then $\mathcal{D}$ has the structure of a $\Z$-graded $\CW_V$-module with
\[
\partial\frac{d^k \delta(x)}{d x^k} = \frac{d^{k+1}\delta(x)}{d x^{k+1}}, \qquad x \frac{d^k\delta(x)}{d x^k} = - k \frac{d^{k-1}\delta(x)}{d x^{k-1}},
\]
where the first relation is by definition, and the second can be verified by using \eqref{eq:derivative-delta}.
Therefore, the linear map  $\C[\partial]\longrightarrow \mathcal{D}$ defined by
\[
\partial ^k \mapsto  \frac{d^k\delta(x)}{d x^k}, \quad k=0, 1, \dots,
\]
gives rise to  an isomorphism of $\Z$-graded $\CW_V$-modules.
\end{example}

This example is a special case of what discussed in \cite[\S 3.2.2, \S 3.2.3]{W}.
The generalisation of this example to mixed Fock spaces of arbitrary Weyl superalgebras is immediate, see \cite[\S 3.2.2, \S 3.2.3]{W} for a detailed discussion.

Now we consider representations of the Weyl superalgebra $\CW_V$.  Let $D\subset V$ be a sub superspace, and let
$D^\perp =\{{\bar v}\in V^*\mid {\bar v}(w)=0, \ \forall\, w\in D\}$.
\begin{lemma} \label{lem:F-gen}
Corresponding to each sub superspace $D$ of $V$, there exists  a simple left $\CW_V$-module $\CF_V(D)$ defined by
\[
\CF_{V}(D)=\frac{\CW_{V}}{\sum\limits_{\partial\in D^\perp} \CW_{V} \partial+ \sum\limits_{v\in D} \CW_{V} v}.
\]
This is a $\Z$-graded $\CW_V$-module which is also $\Z_2$-graded.
Furthermore, two  modules $\CF_V(D)$ and $\CF_V(D')$ are not isomorphic if $D_0\ne D'_0$, where $D_0$ and $D'_0$ are the even subspaces of $D$ and $D'$ respectively.
\end{lemma}
We call $\CF_{V}(D)$ a  {\em mixed Fock space} for $\CW_V$, and call $D$ the {\em mixing set} of $\CF_{V}(D)$.
The usual Fock space for $\CW_V$ is $\CF_{V}(0)$, and the {\em dual Fock space} is
$\CF_{V}(V)$.

\subsubsection{Comments on parity reversal}\label{sect:symm-skew-Pi}

Let $\Pi{V}$ be the parity reversed vector superspace of $V$.
Then as  $\Z$-graded associative algebras, $S(V)$ and  $\bigwedge(V)$ satisfy
\begin{eqnarray}\label{eq:duality-1}
\bigwedge(V) = S(\Pi V), \quad S(V)=\bigwedge(\Pi V).
\end{eqnarray}
Recall from Section \ref{sect:superspaces} that $S(V)$ and $\bigwedge(V)$ inherit $\Z_2$-gradings from $T(V)$,  thus are superalgebras. There is now a slight complication in that
$\bigwedge(V)$ has a different $\Z_2$-grading from that of $S(\Pi V)$, but this is not a disaster.

The $\Z$-graded algebra isomorphism between $\bigwedge(V)$ and $ S(\Pi V)$ can be interpreted as a superalgebra isomorphism,  which however is not homogeneous.  Let  $\pi: V\longrightarrow \Pi{V}$  be the canonical odd isomorphism, which sends any homogeneous element of $V$ to the same element in $\Pi V$ but with the opposite parity. This induces an isomorphism of superalgebras $T(V) \longrightarrow T(\Pi V)$ which is of degree $\alpha$ on the subspace $\oplus_{k=0}^\infty T(V)_{2k+\alpha}$ for $\alpha=0, 1$.  It in turn induces the superalgebra isomorphism between $\bigwedge(V)$ and $S(\Pi V)$, which is what we are after.

We may define a Clifford superalgebra by $\Cl_V = T(V\oplus V^*)/\tilde{J}$, where $\tilde{J}$ is the two-sided ideal generated by the elements
\[
v\otimes w + \tau_{V, V}(v\otimes w), \quad
{\bar v}\otimes {\bar w} + \tau_{V^*, V^*}({\bar v}\otimes \bar{w}), \quad
{\bar v}\otimes w + \tau_{V^*, V}({\bar v}\otimes w) - {\bar v}(w),
\]
for $v, w\in V$ and ${\bar v},  {\bar w}  \in V^*$. Then   as $\Z$-graded algebras,
\begin{eqnarray}\label{eq:duality-2}
\Cl_V=\CW_{\Pi{V}}, \quad \CW_V=\Cl_{\Pi{V}}
\end{eqnarray}
by noting that $(\Pi{V})^*=\Pi{V^*}$.  Discussions above on $\Z_2$-gradings also apply here.

\begin{remark}
The dualities \eqref{eq:duality-1} and \eqref{eq:duality-2} require the introduction of morphisms of associative superalgebras which are more general than those usually allowed. Recall that usually one only allows even morphisms, that is, morphisms which are homogeneous of degree $0$, in a category of superalgebras of any given type.
\end{remark}

\subsection{Differential-integral forms in the Lie superalgebra context}\label{sect:forms}

Fix a Lie superalgebra $\fg=\fg_0\oplus\fg_1$, and denote by $\fg^*$ its dual space. Let $\Pi\fg$ (resp. $\Pi\fg^*$) be the image of $\fg$ (resp. $\fg^*$) under the parity reversal functor.  Note that  $\Pi\fg^*=(\Pi\fg)^*$.

We regard $\fg$ as the Lie superalgebra of the left invariant vector fields on a Lie supergroup $G$. Then forms on $G$ are polynomial functions and distributions on $\Pi\fg$. In particular, differential forms can be identified with elements of $S(\Pi\fg^*)$.

Following \cite[\S3.2]{W}, we can reformulate forms on $G$ using modules over the Weyl superalgebra $\CW_{\Pi\fg^*}$ over $\Pi\fg^*$. The differential-integral forms now correspond to elements of mixed Fock spaces $\CF_{\Pi\fg^*}(D)$ for $\CW_{\Pi\fg^*}$. In particular, the vector superspace of differential forms now corresponds to the standard Fock space $\CF_{\Pi\fg^*}(0)=S(\Pi\fg^*)$.

\subsection{Realisations of Lie superalgebras}

The dual vector superspace $\fg^*$ of the Lie superalgebra $\fg$ has a $\fg$-module structure with the $\fg$-action defined, for any $A\in\fg$ and $v^*\in \fg^*$,  by
\[
A.v^* (B) = (-1)^{[A][v^*]}v^*([A, B]), \quad \forall\, B\in\fg.
\]
There is an odd (i.e., degree $1$) vector superspace isomorphism $\pi: \fg^* \longrightarrow \Pi\fg^*$.
Now $\Pi\fg^*$ has the following $\fg$-module structure:
\begin{eqnarray}\label{eq:Pi-g}
A.\pi(v^*) = (-1)^{[A]} \pi(A.v^*), \quad \forall\, A\in\fg, \ v^*\in \fg^*.
\end{eqnarray}

Write $\CW=\CW_{\Pi\fg^*}$ for the Weyl superalgebra over $\Pi\fg^*$.

To describe $\CW$ more explicitly,   we assume that $\fg$ has super dimension $\sdim\fg=(N|M)$,  that is,
$\dim\fg_0=N$ and $\dim\fg_1=M$.
We choose a basis
 $\{E_i\mid 1\le i\le M\}$ for $\fg_1$, and a basis $\{E_{M+j}\mid 1\le j\le N\}$ for $\fg_0$. Here $M+1, M+2, \dots$ are merely symbols, thus $M$ and $N$ may be infinite.
Let $\Lambda=\{i\mid 1\le i\le M\}\cup \{M+j\mid 1\le j\le N\}$.
Now $\{E_a\mid a\in\Lambda\}$ is a homogeneous basis of $\fg$.
We denote by $ad$ the adjoint representation of $\fg$ relative to this basis,  that is,
\[
{[A, E_b]} = \sum_{c\in\Lambda} \ad(A)_{c b} E_c, \quad \forall\, b.
\]
Let $\{X_a\mid a\in\Lambda\}$ be the homogeneous basis of $\Pi\fg^*$ dual to the basis $\{E_a\mid a\in\Lambda\}$ of $\fg$.  Note that $X_i$ are odd for $1\le i\le N$ and $X_{N+j}$ are even for $1\le j\le M$. Then the Weyl superalgebra $\CW$ is generated by $X_a$, $\partial_a(:=\partial_{X_a})$ with $a\in\Lambda$ subject to the usual relations, namely,
\[
\begin{aligned}&
X_a X_b - (-1)^{[X_a][X_b]}X_b X_a =0, \quad  \partial_a \partial_b - (-1)^{[X_a][X_b]}\partial_b \partial_a =0, \\
&
\partial_a X_b - (-1)^{[X_a][X_b]}X_b \partial_a =\delta_{a b} \quad \text{ for all $a, b\in\Lambda$}.
\end{aligned}
\]
\begin{remark} Note that $[X_a]-[E_a]\equiv 1\,({\rm mod}\, 2)$ for all $a$.
\end{remark}

We write the $\Z_2$-graded commutator in $\CW$ as
\begin{eqnarray}
[A, B]=A B - (-1)^{[A][B]} B A, \quad A, B\in \CW.
\end{eqnarray}
Denote by $\Der_0(S(\Pi\fg^*))$ the sub superspace of $\CW$ consisting of $\Z_2$-graded derivations of
$S(\Pi\fg^*)$, that is,
\[
\Der_0(S(\Pi\fg^*))=\big\{\Xi\in\CW\ \big|\ \Xi(f g) = \Xi(f )g + (-1)^{[\Xi][f]}f \Xi(g), \ f, g\in S(\Pi\fg^*) \big\}.
\]
It is a Lie superalgebra with the super Lie bracket being the graded commutator in $\CW$.
We have the following result.
\begin{lemma}  \label{lem:realise}
There exists a homomorphism of Lie superalgebras defined by
\begin{align}
&\Gamma: \fg\longrightarrow \Der_0(S(\Pi\fg^*)), \quad
\Gamma(A)=- \sum_{b, c\in\Lambda} (-1)^{[A][E_c]} X_c \ad(A)_{b c} \partial_b.
\end{align}
The image $\Gamma(\fg)$ acts on  $\partial_{\fg^*} =\sum_{a\in\Lambda}\C \partial_a$ $($resp.  $\fg^*=\sum_{a\in\Lambda}\C X_a)$ by the adjoint representation $($resp. the dual of the adjoint representation$)$ of $\fg$. Explicitly, for all $A\in\fg$,
\begin{eqnarray}
&&{[\Gamma(A), X_b]} = -\sum_{c\in\Lambda} (-1)^{[A][E_c]}X_c \ad (A)_{b c},\label{eq:rep-X}\\
&&{[\Gamma(A), \partial_b]}= (-1)^{[A]}\sum_{c\in\Lambda} \ad (A)_{c b}\partial_c,
 \quad b\in\Lambda.  \label{eq:rep-d}
\end{eqnarray}
\end{lemma}
\begin{proof}  The following very easy computation proves \eqref{eq:rep-X}.
\[
\begin{aligned}
{[\Gamma(A), X_b]}  &=  - \sum_{a, c\in\Lambda} (-1)^{[A][E_c]} X_c \ad (A)_{a c} [\partial_a, X_b]\\
&=
- \sum_{c\in\Lambda} (-1)^{[A][E_c]} X_c \ad (A)_{b c}.
\end{aligned}
\]
  A similar computation proves \eqref{eq:rep-d}.

It follows from \eqref{eq:rep-X} that
\[
\begin{aligned}
{[\Gamma(B), [\Gamma(A), X_b]]}  &=  \sum_{a, c\in\Lambda} (-1)^{[A][E_c] + [B][E_a]} X_a \ad (B)_{c a}  \ad (A)_{b c}\\
&=  (-1)^{[A][B]}  \sum_{a, c\in\Lambda} (-1)^{([A]+ [B])[E_a]} X_a  (\ad (A) \ad (B))_{b a},
\end{aligned}
\]
and similarly,
\[
\begin{aligned}
{[\Gamma(A), [\Gamma(B), X_b]]}
&=  (-1)^{[A][B]}  \sum_{a, c\in\Lambda} (-1)^{([A]+ [B])[E_a]} X_a  (\ad (B) \ad (A))_{b a}.
\end{aligned}
\]
Now ${[[\Gamma(A), \Gamma(B)], X_b]} = {[\Gamma(A), [\Gamma(B), X_b]]}   -   (-1)^{[A][B]}  {[\Gamma(B), [\Gamma(A), X_b]]}$, hence
\[
\begin{aligned}
{[[\Gamma(A), \Gamma(B)], X_b]}
&= (-1)^{[A][B]}  \sum_{a, c\in\Lambda} (-1)^{([A]+ [B])[E_a]} X_a  \\
&\qquad \left((\ad (B) \ad (A))_{b a} -
(-1)^{[A][B]} (\ad (A) \ad (B))_{b a}\right)\\[6pt]
&=-\sum_{a, c\in\Lambda} (-1)^{([A]+ [B])[E_a]} X_a [\ad (A), \ad (B)]_{b  a} \\
&=-\sum_{a, c\in\Lambda} (-1)^{([A]+ [B])[E_a]} X_a \ad ([A,  B])_{b  a}.
\end{aligned}
\]
As $[\Gamma(A), \Gamma(B)]$ is linear in $X_a$'s and in $\partial_a$'s, this implies that
\[
\begin{aligned}
{[\Gamma(A), \Gamma(B)]}
&=-\sum_{a, c\in\Lambda} (-1)^{([A]+ [B])[E_a]} X_a \ad ([A,  B])_{b  a}\partial_b = \Gamma([A, B]),
\end{aligned}
\]
proving that $\Gamma$ is a Lie superalgebra homomorphism.
This completes the proof.
\end{proof}

We note in particular that for all $a, b\in\Lambda$,
\begin{eqnarray}
&&\Gamma(E_a)=\Gamma_a := -\sum_{b,c\in\Lambda} (-1)^{[E_a][E_c]}X_c \ad (E_a)_{b c} \partial_b,  \\
&&{[\Gamma_a, X_b]} = \sum_{c\in\Lambda} X_c \ad (E_c)_{b a}, \quad
[\Gamma_a, \partial_b] = (-1)^{[E_a]}\sum_{c\in\Lambda} \ad (E_a)_{c b}\partial_c, \label{eq:ad-reps}
\end{eqnarray}
and $\Gamma_a = \sum_{b,c\in\Lambda} X_c \ad (E_c)_{b a} \partial_b$.

The following result is an immediate consequence of Lemma \ref{lem:realise}.
\begin{corollary} \label{cor:action-W}
The Weyl superalgebra $\CW_{\Pi\fg^*}$ admits the following $\fg$-action
\[
\fg\times \CW_{\Pi\fg^*}\longrightarrow \CW_{\Pi\fg^*}, \quad (A, f)\mapsto A(f):=[\Gamma(A), f], \quad \forall\, A\in \fg, \ f\in \CW_{\Pi\fg^*}.
\]
That is, the above defines a $\fg$-action on $\CW_{\Pi\fg^*}$, which preserves the superalgebra structure of the latter in the sense that
\[
A(f g) = A(f) g + (-1)^{[A][f]}f A(g), \quad A\in \fg, \ f,  g\in \CW_{\Pi\fg^*}.
\]
\end{corollary}

\subsection{BRST cohomology of Lie superalgebras}\label{sect:BRST}

The Lie superalgebra cohomology given in Definition \ref{def:standard} can be reformulated in terms of representations of Weyl superalgebras in the spirit of BRST theory (see, e.g., \cite{BBH}).

Retain notation in the last section, in particular, $\CW=\CW_{\Pi\fg^*}$.  For any sub superspace $D\subset \Pi\fg^*$, we write $\CF(D)$ for the mixed Fock superspace $\CF_{\Pi\fg^*}(D)$ for simplicity.  Recall that  $\CF(D)=\oplus_{p\in\Z}\CF(D)_p$ is $\Z$-graded.
If the even subspace of $D$ is a proper subspace of $V_0$,  then $\CF(D)_p$ is infinite dimensional for all $p\in\Z$.

We consider $\CW\otimes \U(\fg)$ as a $\Z$-graded superalgebra with $1\otimes \U(\fg)$ in degree $0$ and and thus $(\CW\otimes \U(\fg))_q=\CW_q\otimes \U(\fg)$ for all $q\in\Z$.  For any $\fg$-module $V$ and any mixed Fock superspace  $\CF(D)$ for $\CW$,  the tensor product $C_D(\fg, V):=\CF(D)\otimes V$ is a $\CW\otimes \U(\fg)$ module with a $\Z$-grading given by
\begin{eqnarray} \label{eq:D-chains}
C_D(\fg, V)=\bigoplus_{p\in\Z} C_D^p(\fg, V), \quad
C_D^p(\fg, V) = \CF(D)_p\otimes V.
\end{eqnarray}
It also inherits a $\Z_2$-grading from the $\Z_2$-gradings of $\CF(D)$ and $V$.

Define the following elements of $\CW\otimes \U(\fg)$
\begin{eqnarray}
&&\delta_1 = \sum_{a\in\Lambda} (-1)^{[E_a]} X_a \otimes E_a, \quad  \delta_2 = \frac{1}{2}\sum_{a\in\Lambda} (-1)^{[E_a]}X_a \Gamma_a\otimes 1, \\
&&\delta = \delta_1+\delta_2, \label{eq:delta-operator}
\end{eqnarray}
which are homogeneous of degree $1$ with  respect to the $\Z$-grading, and are odd with respect to the $\Z_2$-grading, of $\CW\otimes\U(\fg)$. Thus as linear operators on $C_D(\fg, V)$, they map $C_D^p(\fg, V)_\varepsilon$ to $C_D^{p+1}(\fg, V)_{\varepsilon+1}$ for any $p\in\Z$ and $\varepsilon\in\Z_2:=\Z/2\Z$.

\begin{remark}
One can easily show that the definitions of $\delta_1$, $\delta_2$,  and hence of $\delta$, are independent of the choice of basis for $\fg$.
\end{remark}

The following result easily follows from Lemma \ref{lem:realise}.
\begin{lemma} \label{lem:op-inv}
The operator $\delta$ defined by \eqref{eq:delta-operator}  is an odd $\fg$-invariant in $\CW\otimes\U(\fg)$, that is, it is of degree $1$ with respect to the $\Z_2$-grading and satisfies
\[
{[\Gamma(A)\otimes 1 + 1\otimes A, \delta]}=0, \quad  \forall\, A\in\fg.
\]
\end{lemma}
\begin{proof} It is clear that $\delta$ is odd.
By Lemma \ref{lem:realise}, $\Gamma(\fg)$ is isomorphic to the adjoint module for $\fg$, and
$\Pi\fg^*\subset\CW$ to the dual  adjoint module with reversed parity. Therefore, $\delta_1$ and $\delta_2$ are both $\fg$-invariant, and hence so is also $\delta$.

We can also verify this by direct computation. We have
\[
\begin{aligned}
&{[\Gamma(A)\otimes 1 + 1\otimes A, \delta_1]}\\
&\phantom{XXX}= \sum_a (-1)^{[E_a]} \left([\Gamma(A), X_a]\otimes E_a + (-1)^{[A]([E_a]+1)} X_a\otimes [A, E_a]\right)\\
&\phantom{XXX}=-\sum_{a, c} (-1)^{[E_a] + [A][E_c]}X_c \otimes \ad (A)_{a c}E_a\\
&\phantom{XXX}\quad +\sum_{a, c} (-1)^{[E_a] + [A]([E_a]+1)} X_a\otimes \ad (A)_{c a}E_c\\
&\phantom{XXX}=0.
\end{aligned}
\]
This shows the invariance of $\delta_1$.
We can also show that $[\Gamma(A), \sum_a (-1)^{[E_a]}X_a \Gamma_a]=0$ by a similar computation. This  implies the $\fg$-invariance of $\delta_2$.
\end{proof}

We have the following lemma, which is of crucial importance for the remainder of this paper.
\begin{lemma} \label{lem:key}
The operators $\delta_1$ and $\delta_2$ satisfy the following relations.
\[
\begin{aligned}
& [\delta_1, \delta_2]+[\delta_2, \delta_1] =-[\delta_1, \delta_1], \quad
&[\delta_2, \delta_2]=0.
\end{aligned}
\]
Hence it follows that the operator $\delta = \delta_1 + \delta_2$ satisfies
\[
(\delta)^2 =0.
\]
\end{lemma}
\begin{proof} The proof is by direct computation, which is relatively straightforward, but we nevertheless present details of the proof because of the importance of the lemma.
Note that we only need to prove the first part of the lemma, as the second part follows from the first.

Let us start by showing that
\begin{eqnarray}\label{eq:d1-d1}
{[\delta_1, \delta_1]}=- \sum_{a, b\in\Lambda} (-1)^{[E_b]}X_a X_b\otimes [E_b, E_a].
\end{eqnarray}
Clearly $[\delta_1, \delta_1]$ is equal to
\[
\begin{aligned}
\sum_{a, b\in\Lambda} (-1)^{[E_a]+[E_b]}\left(X_a X_b\otimes E_a E_b (-1)^{[X_b][E_a]} + X_b X_a\otimes E_b E_a (-1)^{[X_a][E_b]}\right).
\end{aligned}
\]
Since $X_a X_b = (-1)^{[X_a][X_b]} X_b X_a$ and $[X_c]-[E_c]\equiv 1 \,({\rm mod}\, 2)$ for all $c$, we have
\[
{[\delta_1, \delta_1]}= \sum_{a, b\in\Lambda} (-1)^{[X_b]}X_a X_b\otimes[E_b, E_a] =-\sum_{a, b\in\Lambda} (-1)^{[E_b]}X_a X_b\otimes[E_b, E_a].
\]
This proves equation \eqref{eq:d1-d1}.

We note that $[\delta_1, \delta_2]=[\delta_2, \delta_1]$.  Now
\[
2{[\delta_1, \delta_2]}= \sum_{a, b}  (-1)^{[E_a]+[E_b]} X_a[\Gamma_a, X_b]\otimes E_b.
\]
Using the first relation in \eqref{eq:ad-reps},  we obtain
\[
\begin{aligned}
2{[\delta_2, \delta_1]}&= \sum_{a, b, c} (-1)^{[E_a]+[E_b]} X_a X_c\otimes  \ad (E_c)_{b a} E_b\\
				&= \sum_{a, b, c} (-1)^{[E_c]} X_a X_c\otimes  \ad (E_c)_{b a} E_b\\
				&= \sum_{a, c} (-1)^{[E_c]} X_a X_c\otimes  [E_c, E_a].
\end{aligned}
\]
Hence $[\delta_1, \delta_2] + [\delta_2, \delta_1]= -[\delta_1, \delta_1]$ by  \eqref{eq:d1-d1}, proving the first relation in the lemma.

Now consider $[\delta_2, \delta_2]$. The property $[A, BC]=[A, B] C + (-1)^{[A][B]}B[A, C]$ of the super commutator leads to
\[
\begin{aligned}
{[\delta_2, \delta_2]}&= \frac{1}{2}\sum_{a, b} (-1)^{[E_a]+[E_b]}X_a[\Gamma_a, X_b]\Gamma_b\\
&\phantom{=}+ \frac{1}{4} \sum_{a, b} (-1)^{([E_a]+1)[E_b]}X_aX_b[\Gamma_a, \Gamma_b].
\end{aligned}
\]
Using Lemma \ref{lem:realise}, we obtain
\[
\begin{aligned}
{[\delta_2, \delta_2]}&= - \frac{1}{4} \sum_{a, b} (-1)^{([E_a]+1)[E_b]}X_aX_b[\Gamma_a, \Gamma_b]\\
&=- \frac{1}{4} \sum_{a, b} (-1)^{([E_a]+1)[E_b]}X_a X_b \ad (E_a)_{c b} \Gamma_c.
\end{aligned}
\]
As this operator is tri-linear in the elements $X_f$ and linear in the elements $\partial_f$, it vanishes
if and only if  $[\delta_2, \delta_2](X_f)=0$ for all $f\in\Lambda$.  Now by Lemma \ref{lem:realise},
\[
\begin{aligned}
{[\delta_2, \delta_2]}(X_f)
&=\frac{1}{4} \sum_{a, b, c, g} (-1)^{([E_a]+1)[E_b] +[E_c][E_g]}X_a X_b X_g \ad (E_a)_{c b}  \ad (E_c)_{f g}.
\end{aligned}
\]
Introduce the element $\Theta:= \sum_{f\in\Lambda}{[\delta_2, \delta_2]}(X_f)\otimes E_f$ in $\CW\otimes\fg$. Then
$[\delta_2, \delta_2]=0$ if and only if $\Theta=0$.  We have
\[
\begin{aligned}
\Theta
&=\frac{1}{4} \sum_{a, b, g} (-1)^{([E_a]+1)[E_b] +([E_a]+[E_b])[E_g]}X_a X_b X_g \otimes [[E_a, E_b], E_g]\\
&=\frac{1}{4} \sum_{a, b, c} (-1)^{[E_b]} X_a X_b X_c \otimes [E_c, [E_b, E_a]].
\end{aligned}
\]
The right hand side can clearly be re-written as
\[
\begin{aligned}
\frac{1}{12} \sum_{a, b, c}\Big(
&(-1)^{[E_b]} X_a X_b X_c \otimes [E_c, [E_b, E_a]]\\
&+ (-1)^{[E_c]}  X_b X_c X_a\otimes [E_a, [E_c, E_b]]\\[6pt]
&+ (-1)^{[E_a]} X_c X_a X_b \otimes [E_b, [E_a, E_c]]\Big).
\end{aligned}
\]
By the super commutativity property of $X_f$,  this can be further re-written as
\[
\begin{aligned}
\frac{1}{12} \sum_{a, b, c}& (-1)^{[E_b]} X_a X_b X_c \otimes \Big ([E_c, [E_b, E_a]]\\
&+ (-1)^{[E_a]([E_b]+[E_c])} [E_a, [E_c, E_b]]\\[6pt]
 &+(-1)^{[E_c]([E_a]+[E_b])} [E_b, [E_a, E_c]]\Big),
\end{aligned}
\]
which vanishes identically by the Jacobian identity of $\fg$.

This completes the proof.
\end{proof}

 The following result is an obvious corollary of Lemma \ref{lem:key}.
\begin{theorem} \label{thm:main}
Given a sub superspace $D\subset\Pi\fg^*$ and a $\fg$-module $V$,  set
$C_D^p(\fg, V)$ $=$ $\CF(D)_p\otimes V$ for all $p\in\Z$ $($cf. \eqref{eq:D-chains}$)$. Let $\delta$ be the operator defined by \eqref{eq:delta-operator}. Then there is the following
 differential complex
\begin{eqnarray*}
\cdots\stackrel{\delta}\longrightarrow C_D^{-2}(\fg, V)\stackrel{\delta}\longrightarrow C_D^{-1}(\fg, V)\stackrel{\delta}\longrightarrow  C_D^0(\fg, V)\stackrel{\delta}\longrightarrow  C_D^1(\fg, V)\stackrel{\delta}\longrightarrow\cdots,
\end{eqnarray*}
which will be denote by $(C_D(\fg, V), \delta)$, and its homology groups by $H^p_D(\fg, V)$.
\end{theorem}
\begin{proof}
It follows Lemma \ref{lem:key} that $(C_D(\fg, V), \delta)$ is indeed a differential complex.
\end{proof}

\begin{remark}
If $\fg$ is an ordinary Lie algebra, then $\Pi\fg^*$ is purely odd, and $\CW_{\Pi\fg^*}$ is a Clifford algebra. Therefore  $\CF_{\Pi\fg^*}(0)=\CF_{\Pi\fg^*}(\Pi\fg^*)$, and hence $H^{-p}_{\Pi\fg^*}(\fg, V)=H^{\dim\fg-p}_0(\fg, V)$ for all $p=0, 1, \dots, \dim\fg$.
\end{remark}

\subsection{Mixed cohomology of Lie superalgebras}
\subsubsection{BRST formulation of the standard cohomology}\label{sect:BRST-old}

We now show that  the standard cohomology of Lie superalgebras given in Definition \ref{def:standard} is the special case of the cohomology in Theorem \ref{thm:main} with $D=0$.
We write $\CF=\CF(0)$, then $\CF=S(\Pi\fg^*)\cong {\bigwedge}\fg^*$.  By Section \ref{sect:symm-skew-Pi},
the vector superspace isomorphisms $\CF_p\cong {\bigwedge}^p\fg^*$ are even (resp., odd) for  even (resp., odd) $p$. Thus we have the isomorphisms  of vector superspaces
\[
\iota_p: C^p(\fg, V) \stackrel{\cong}\longrightarrow \CF_p\otimes V, \quad \forall\, p,
\]
which are even (resp., odd) for even  (resp., odd) $p$.
\begin{theorem} \label{lem:reform} \label{thm:delta}  \label{thm:reform}
\begin{enumerate}
\item Define the following linear maps
\begin{eqnarray}
\delta^p:= \iota_{p+1}d^p\iota_p^{-1}:  \CF_p\otimes V  \longrightarrow \CF_{p+1}\otimes V \quad \text{for all $p$}.
\end{eqnarray}
Then there is a  differential complex
\[
(\CF\otimes V, \delta):\qquad \CF_0\otimes V \stackrel{\delta^0 }\longrightarrow \CF_1\otimes V \stackrel{\delta^1 }\longrightarrow  \CF_2\otimes V \stackrel{\delta^2 }\longrightarrow  \cdots,
\]
which is isomorphic to the differential complex $(C(\fg, V), d)$ in Definition $\ref{def:standard}$.

\item
The restriction of $\delta_i$ $(i=1, 2)$ to $\CF_p\otimes V$ coincides with
$\iota_p d_i \iota_p^{-1}$ for all $p$, and hence
\begin{eqnarray}
\delta^p = \delta|_{\CF_p\otimes V}, \quad \forall\, p.
\end{eqnarray}
\end{enumerate}
\end{theorem}
\begin{proof} Part (1) is clear, thus we only need to prove part (2).
By inspecting the structure of the operators $\delta_1$ and $\delta_2$, one can see
that the theorem is valid if $\fg$ is an ordinary Lie algebra, that is, a purely even Lie superalgebra.  This is a well known fact from the study of BRST cohomology of Lie algebras in the physics literature.  To generalise this to Lie superalgebras with non-trivial odd subspaces, we only need to check that the operators lead to the correct signs in \eqref{eq:d-1} and \eqref{eq:d-2}, and it is indeed true.
\end{proof}

We may restate the above theorem as follows, obtaining
a BRST reformulation of the standard cohomology of Lie superalgebras  \cite{FL, SchZ} given in Definition \ref{def:standard}.
\begin{corollary} \label{cor:reformulation}
When $D=0$,  the differential complex $(C_0(\fg, V), \delta)$ of Theorem \ref{thm:main}
is isomorphic to the complex of $\fg$-modules in Definition $\ref{def:standard}$.
\end{corollary}
Note that the superspace isomorphism of cochains of degree $p\in\Z$ is even (resp., odd) if $p$ is even (resp., odd) with respect to the $\Z_2$-grading.

\begin{remark}
Lemma \ref{lem:op-inv} and also Lemma \ref{lem:mod-inv}  below describe explicitly the $\fg$-module structure of the differential complex $(C_0(\fg, V), \delta)$.
\end{remark}

\subsubsection{Mixed cohomology of Lie superalgebras}

Recall the $\fg$-action on the Weyl superalgebra $\CW$ given in Lemma \ref{cor:action-W}.
If $D\subset \Pi\fg^*\subset\CW$ is a $\fg$-submodule, then so is also
$D^\perp =\{\partial\in (\Pi\fg^*)^*\mid \partial(w)=0, \ \forall\, w\in D\}$.
It follows that the left ideal
\begin{eqnarray}\label{eqn:ideal-1}
J(D):= \sum\limits_{\partial\in D^\perp} \CW \partial+ \sum\limits_{x\in D} \CW x
\end{eqnarray}
is also a  $\fg$-submodule of $\CW$, and hence $\CF(D)=\CW/J(D)$ is a quotient module. The following result is immediate.

\begin{lemma}\label{lem:mod-inv}
Assume that the mixing set $D$ of a $\CW$-module $\CF(D)$ is a $\fg$-submodule of $\Pi\fg^*$. Then for any $\fg$-module $V$, the superspace of mixed cochains $C_D(\fg, V)=\CF(D)\otimes V$ has the structure of a $\Z$-graded $\fg$-module.
\end{lemma}

The following result is an obvious corollary of Lemma \ref{lem:mod-inv} and Lermma \ref{lem:op-inv}.
\begin{theorem}
Assume that the mixing set $D$ of a $\CW$-module $\CF(D)$ is a $\fg$-submodule of $\Pi\fg^*$. Then $\delta: C_D(\fg, V)\longrightarrow C_D(\fg, V)$ is an odd $\fg$-module homomorphism which has degree $1$ with respect to the $\Z$-grading of $C_D(\fg, V)$, and in this sense the differential complex $(C_D(\fg, V), \delta)$ of Theorem $\ref{thm:main}$ is a complex of $\fg$-modules. Denote by $H^*_D(\fg, V)$ its homology.
\end{theorem}

This enables us to make the following definition.

\begin{definition}\label{def:mixed}
Retain the setting of the theorem above. Call $(C_D(\fg, V), \delta)$ a mixed complex of $\fg$-modules, and its homology a mixed cohomology of $\fg$ with coefficients in $V$.
\end{definition}

\begin{remark}
As we have already seen, in the special case with $D=0$, we have $C_0(\fg, V)= \CF(0)\otimes V$, and the complex $(C_D(\fg, V), \delta)$ is the generalised Chevalley-Eilenberg complex for the standard cohomology of the Lie superalgebra $\fg$.
\end{remark}

\subsubsection{Completion}\label{sect:completion}
Note that if $D=0$ (resp. $D=\Pi\fg^*$), then homogeneous subspaces of $\CF(D)$ are all finite dimensional if $\fg$ is finite dimensional, and there exists no homogeneous subspaces of negative (resp. positive) degrees.
However, if the even subspace of $D$ is a proper subspace of $V_0$,  then $\CF(D)_p$ is infinite dimensional for all $p\in\Z$. Thus there is the possibility to complete  each $\CF(D)_p$ by introducing a topology.

Each $\CF(D)_p$ is filtered by the degree of $\Pi\fg^*$. Explicitly, we let $\CF(D)^k_p$ be the subspace of $\CF(D)_p$ spanned by monomials of $X_a\not\in D$ and $\partial_b\not\in D^\perp$ which are of order $\le k$ in the $X_a$'s (cf.~\eqref{eqn:ideal-1}). Then
\[
\CF(D)_p^0\subset \CF(D)_p^1\subset \CF(D)_p^2\subset \cdots.
\]
We take the complements $\overline{\CF(D)}^k_p$ of
$\CF(D)^k_p$ in $\CF(D)_p$ as the  fundamental system of open neighbourhoods at $0$,  and the arbitrary unions of the finite intersections of $\overline{\CF(D)}^k_p$ as the open sets.
Denote by $\widehat{\CF(D)}_p$ the completion of
$\CF(D)_p$ in this topology.

\begin{lemma}
Retain the notation above.  Assume that the mixing set $D\subset\Pi\fg^*$ is a $\fg$-submodule, and let $\HC_D^p(\fg, V) = \widehat{\CF(D)}_p\otimes V$.  Then there exists the following differential complex of $\fg$-modules.
\begin{eqnarray*}
\dots\stackrel{\delta}\longrightarrow \HC_D^{-2}(\fg, V)\stackrel{\delta}\longrightarrow \HC_D^{-1}(\fg, V)\stackrel{\delta}\longrightarrow  \HC_D^0(\fg, V)\stackrel{\delta}\longrightarrow  \HC_D^1(\fg, V)\stackrel{\delta}\longrightarrow\cdots.
\end{eqnarray*}
Denote the cohomology groups of this complex by $\HH_D^p(\fg, V)$.
\end{lemma}

\section{Examples: mixed $\fu$-cohomologies of $\gl_{m |n}$}\label{sect:compute}

Let $\fg=\gl_{m|n}$, and let $V=\C^{m|n}$ be the natural module for $\gl_{m|n}$.  Choose the standard basis $v_i,\,i\in I:=I_0\cup I_1$ for $V$,
where $I_0=\{1,...,m\},$ $I_1=\{m+1,...,m+n\}$, with $v_i$ being even for $i\in I_0$ and   odd for $i\in I_1$.  Let $\{E_{i j}\mid i, j\in I\}$ be the set of
matrix units relative to this basis, which forms a homogeneous basis of $\gl_{m|n}$.
We will consider mixed $\fu$-cohomologies of $\fg$ for two choices of parabolic subalgebras $\fp=\fl\oplus\fu$ with coefficients in $V$. It is essential to understand $\fu$-cohomology with coefficients in simple $\fg$-modules in order to understand parabolic induction.

Let $\CW=\CW_{\Pi\fu^*}$ be the Weyl superalgebra over $\Pi\fu^*$, the dual superspace of $\fu$ with parity reversed.  We study the mixed $\fu$-cohomologies in the setting of the Weyl superalgebra $\CW$.

%
%
\subsection{The case with Levi subalgebra $\gl_m\oplus\gl_n$} \label{sect:usual-para}

Consider the parabolic subalgebra $\fp:=\gl_m\oplus\gl_n\oplus \fu$ with the nilpotent ideal
$\fu=\oplus_{i\in I_0,\,j\in I_1} \C E_{ij}$.
Let $x_{ij},\,i\in I_0,j\in I_1$ be basis elements of $\Pi\fu^*\subset\CW$ such that
$
\begin{aligned}
\pi^{-1}(x_{ij})(E_{k\ell})=\delta_{ik}\delta_{j\ell}.
\end{aligned}
$
We denote the corresponding derivations by $\partial_{x_{ij}}=\frac{\partial}{\partial x_{ij}}$.
The differential operator for the $\fu$-cohomology is the negative of
\begin{eqnarray}\label{eq:delta-usual-para}
\delta=\sum_{i\in I_0,j\in I_1} x_{ij}\otimes E_{ij}.
\end{eqnarray}

\subsubsection{Standard $\fu$-cohomology}
The superspace of cochains is
$
C(\fu, V) = {\mathcal A}\otimes V,
$
where ${\mathcal A}=\C[x_{ij}\,|\,i\in I_0,j\in I_1]$ is the space of polynomials in $x_{ij}$'s.  Let
$f=\sum_{i\in I}f_{i} v_{i}$ be an arbitrary element in $C(\fu, V)$, where $f_{i}\in {\mathcal A}$. We have
$
\delta(f) = \sum_{i\in I_0}\big(\sum_{j\in I_1}x_{ij}f_j\big)v_i.
$
Thus
\begin{align*}
&\ker \delta = {\mathcal A} v_1 +\cdots+ {\mathcal A}v_m+K_1,
\ \
K_1=\left\{\sum_{j\in I_1}f_jv_j\ \left|\ \sum_{j\in I_1}x_{ij}f_j\!=\!0,\,\forall\,i\in I_0\right\}\right.,\!\!\!\!\!
\\
& \im \delta=\left.\left\{\sum_{i\in I_0}\Big(\sum_{j\in I_1}x_{ij}f_j\Big)v_i\ \right|
\ f_j\in{\mathcal A},\,\forall\,j\in I_1\right\} =\sum_{j\in I_1}{\mathcal A}\Big(\sum_{i\in I_0}x_{ij}v_i\Big); \\
&H^*(\fu, V) = \frac{{\mathcal A} v_1 +\cdots+ {\mathcal A}v_m}{\sum\limits_{j\in I_1}{\mathcal A}\Big(\sum\limits_{i\in I_0}x_{ij}v_i\Big)} + K_1,
\end{align*}
where we note that $K_1=0$ if and only if $m\ge n$.

These  $\fu$-cohomology groups were calculated in \cite{CZ}  using a different method.

\subsubsection{The $\fu$-cohomology on the dual Fock space}
We take $D=\Pi\fu^*$. The space of cochains is given by
$
C_D(\fu, V) ={\mathcal B}\otimes V,
$
where ${\mathcal B}=\C[\partial_{x_{ij}}\,|\,i\in I_0,j\in I_1]$ is the dual Fock space such that the vacuum vector $1$ satisfies
$x_{ij} 1=0$ for $i\in I_0,j\in I_1$. Let $f=\sum_{i\in I}f_i v_i$ be in $C_D(\fu, V)$, where $f_i\in{\mathcal B}$. We have
$
\delta(f)= -
 \sum_{i\in I_0}\big(\sum_{j\in I_1}\partial_{\bar x_{ij}}f_j\big)v_i,
$
where
we have denoted $\bar x_{ij}=\partial_{x_{ij}}$ and $\partial_{\bar x_{ij}}=\frac{\partial}{\partial \bar x_{ij}}$, thus
$\partial_{\bar x_{ij}}f_{j}$ is the partial derivative with respect to $\bar x_{ij}=\partial_{x_{ij}}$ of the polynomial $f_{j}$.
Hence  for $D=\Pi\fu^*$,
\begin{align*}
&\ker \delta = {\mathcal B} v_1+\cdots + {\mathcal B} v_m + K_2,
\ \
K_2=\left\{\sum_{j\in I_1}f_jv_j\,\left|\,\sum_{j\in I_1}\partial_{\bar x_{ij}}f_j\!=\!0,\,\forall\,i\in I_0\right\}\right.,\!\!\!\!\!
\\
& \im \delta = \left.\left\{ \sum_{i\in I_0}\Big(\sum_{j\in I_1}\partial_{\bar x_{ij}}f_j\Big)v_i
\ \right|\ f_j\in{\mathcal B},\,\forall\,j\in I_1\right\}; \\
& H^*_D(\fu, V)= \frac{ {\mathcal B} v_1+\cdots + {\mathcal B} v_m }{\Big\{ \sum\limits_{i\in I_0}\Big(\sum\limits_{j\in I_1}\partial_{\bar x_{ij}}f_j\Big)v_i
\ \Big|\ f_j\in{\mathcal B},\,\forall\,j\in I_1\Big\}}+  K_2.
\end{align*}

\subsubsection{Mixed $\fu$-cohomology}

We take $D=\oplus_{m_1+1\le i\le m,\,j\in I_1}\C x_{i,j}$ for any $m_1\in\Z$ with $1\le m_1<m$. Then the superspace of cochains is
$
C_D(\fu, V) = {\mathcal C}\otimes V,
$
where ${\mathcal C}=\C[x_{ij},\bar x_{kj}\,|\,1\le i\le m_1<k\le m,\,j\in I_1]$ is the space of polynomials in $x_{ij}$'s and $\bar x_{kj}$'s with $\bar x_{ij}=\partial_{x_{ij}}$.
For any $f=\sum_{i\in I} f_i v_i$ with $f_i\in{\mathcal C}$, we have
\[
\delta(f) =
\sum_{i=1}^{m_1}\Big(\sum_{j\in I_1}x_{ij}f_j\Big)v_i
-\sum_{i=m_1+1}^{m}\Big(\sum_{j\in I_1}\partial_{\bar x_{ij}}f_j\Big)v_i
.
\]
Thus
\[
\begin{aligned}
&\ker\delta={\mathcal C}v_1+\cdots+{\mathcal C}v_m+K_3,
\\
&
K_3=\left\{\sum_{j\in I_1}f_jv_j\ \left|\ \sum_{j\in I_1}x_{ij}f_j=\sum_{j\in I_1}\partial_{\bar x_{kj}}f_j
=0,\,1\le i\le m_1<k\le m\right\}\right.,\!\!\!\!\!
\\
&\im\delta=\left.\left\{\sum_{i=1}^{m_1}\Big(\sum_{j\in I_1}x_{ij}f_j\Big)v_i
-\sum_{i=m_1+1}^{m}\Big(\sum_{j\in I_1}\partial_{\bar x_{ij}}f_j\Big)v_i
\ \right|\ f_j\in{\mathcal C},\,\forall\,j\in I_1\right\};
\\[4pt]
&H^*_D(\fu, V) =\frac{{\mathcal C}v_1+\cdots+{\mathcal C}v_m}{\Big\{\sum\limits_{i=1}^{m_1}\Big(\sum\limits_{j\in I_1}x_{ij}f_j\Big)v_i
-\sum\limits_{i=m_1+1}^{m}\Big(\sum\limits_{j\in I_1}\partial_{\bar x_{ij}}f_j\Big)v_i
\ \Big|\  f_j\in{\mathcal C},\,\forall\,j\in I_1\Big\}}+K_3
.
\end{aligned}
\]

%
%
\subsection{The case with Levi subalgebra $\gl_{m_1}\oplus\gl_{m-m_1|n}$}

Consider the parabolic subalgebra $\fp=\fl\oplus\fu$ with $\fl=\gl_{m_1}\oplus\gl_{m-m_1|n}$ and
$\fu=\oplus_{1\le i\le m_1,m_1+1\le j\le m+n}\C E_{ij}$ for any  $m_1\in\Z$ with $1\le m_1<m$.
Then $\Pi\fu^*$ has basis $\theta_{ij},x_{ik}$ with $1\le i\le m_1<j\le m<k\le m+n$. The differential operator is given by
\[
\delta=\sum_{i=1}^{m_1}
\Big(\sum_{j=m_1+1}^m \theta_{ij}\otimes E_{ij} -\sum_{k\in I_1} x_{ik}\otimes E_{ik}\Big).
\]

\subsubsection{Standard $\fu$-cohomology}
%
%
The superspace of cochains is
$
{\mathcal D}\otimes V,
$
where ${\mathcal D}=\C[\theta_{ij},x_{ik}\,|\,1\le i\le m_1<j\le m<k\le m+n]$
is the polynomial ring in $\theta_{ij}$'s and $x_{ik}$'s with $\theta_{ij}$ being odd satisfying $\theta^2_{ij}=0$.
For any $f=\sum_{i\in I}f_iv_i$ with $f_i\in{\mathcal D}$, we have
$
\delta(f)= \sum_{i=1}^{m_1}\big(\sum_{j=m_1+1}^m\theta_{ij}f_j-\sum_{k\in I_1}x_{ik}f_k\big)v_i.
$
Hence
\[
\begin{aligned}
&\ker\delta=
{\mathcal D}v_1+\cdots+{\mathcal D}v_{m_1}+K_4,
\\ &
K_4=\left\{
\sum_{i=m_1+1}^{m+n}f_iv_i\ \left|\ \sum_{j=m_1+1}^m\theta_{ij}f_j-\sum_{k\in I_1}x_{ik}f_k=0,\,1\le i\le m_1
\right\}\right.,
\\
&\im\delta=
\left.\left\{\sum_{i=1}^{m_1}\Big(\sum_{j=m_1+1}^m\theta_{ij}f_j-\sum_{k\in I_1}x_{ik}f_k\Big)v_i\ \right|\ f_j\in{\mathcal D},\,m_1<j\le m+n\right\}
, \\[4pt]
&H^*(\fu, V)
=\frac{{\mathcal D}v_1+\cdots+{\mathcal D}v_{m_1}}{\Big\{\sum\limits_{i=1}^{m_1}\Big(\sum\limits_{j=m_1+1}^m\theta_{ij}f_j-\sum\limits_{k\in I_1}x_{ik}f_k\Big)v_i\ \Big|\ f_j\in{\mathcal D},\,m_1<j\le m+n\Big\}}+K_4
.
\end{aligned}
\]

\subsubsection{The $\fu$-cohomology on dual Fock space}
The superspace of cochains is
$
{\mathcal E}\otimes V,
$ where ${\mathcal E}=\C[\bar \theta_{ij},\bar x_{ik}\,|\,1\le i\le m_1<j\le m<k\le m+n]$ is the polynomial ring on $\bar\theta_{ij}$'s and $\bar x_{ij}$'s
with $\bar\theta_{ij}=\partial_{\theta_{ij}}$ and $\bar{x}_{ik}=\partial_{x_{ik}}$.
For any $f=\sum_{i\in I}f_iv_i$ with $f_i\in{\mathcal E}$, we have
$
\delta(f)= \sum_{i=1}^{m_1}\big(\sum_{j=m_1+1}^m\partial_{\bar \theta_{ij}}f_j-\sum_{k\in I_1}\partial_{\bar x_{ik}}f_k\big)v_i,
$
where $\partial_{\bar \theta_{ij}}=\frac{\partial}{\partial \bar\theta_{ij}}$, $\partial_{\bar x_{ik}}=\frac{\partial}{\partial \bar x_{ik}}$.
Hence
\[
\begin{aligned}
&\ker\delta=
{\mathcal D}v_1+\cdots+{\mathcal D}v_{m_1}+K_5,
\\ &
K_5=\left\{
\sum_{i=m_1+1}^{m+n}f_iv_i\ \left|\ \sum_{j=m_1+1}^m\partial_{\bar \theta_{ij}}f_j-\sum_{k\in I_1}\partial_{\bar x_{ik}}f_k=0,\,1\le i\le m_1
\right\}\right.,
\\
&\im\delta=
\left.\left\{\sum_{i=1}^{m_1}\Big(\sum_{j=m_1+1}^m\partial_{\bar\theta_{ij}}f_j-\sum_{k\in I_1}\partial_{\bar x_{ik}}f_k\Big)v_i\ \right|\ f_j\in{\mathcal D},\,m_1<j\le m+n\right\}
, \\[4pt]
&H^*(\fu, V)
=\frac{{\mathcal D}v_1+\cdots+{\mathcal D}v_{m_1}}{\Big\{\sum\limits_{i=1}^{m_1}\Big(\sum\limits_{j=m_1+1}^m\partial_{\bar\theta_{ij}}f_j-\sum\limits_{k\in I_1}\partial_{\bar x_{ik}}f_k\Big)v_i\ \Big|\ f_j\in{\mathcal D},\,m_1<j\le m+n\Big\}
}+K_5.
\end{aligned}
\]

\subsection{The $\fu$-cohomologies for Kac modules}

In this section we take the parabolic subalgebra of $\gl_{m|n}$ given in Section \ref{sect:usual-para}, and let $K=K_\lambda=\U(\gl_{m|n})\otimes_{\U(\fp)} L_\lambda^0$ be the Kac module with a typical highest weight $\lambda$, where $L_\lambda^0$ is the simple $\fp$-module with highest weight $\lambda$. We consider the mixed $\fu$-cohomologies with coefficients in $K$.  Let $\overline{K}^0_\lambda$ be the lowest component of $K$, then $K$ is a free module over $\U(\fu)$ and we have $K= \U(\fu)\otimes \overline{K}^0_\lambda$. Note that $\U(\fu)=\C[E_{i j}\mid i\in I_0, j\in I_1]$, where $E_{i j}$'s are regarded as Grassmannian variables.

The differential is still given by \eqref{eq:delta-usual-para}.

\subsubsection{The case of the standard Fock space}
The $\fu$-cohomology in this case can be deduced from the generalised Kazhdan-Lusztig theory for parabolic induction of $\gl_{m|n}$ \cite{B, CZ, S}; this is not difficult, but is very indirect.  Within the BRST framework adopted here,  we can read off the cohomology quite directly from a $\gl_{1|1}$-action on the superspace of cochains. This $\gl_{1|1}$-structure is interesting in itself.

Now the superspace of cochains is given by
\[
C(\fu, K) = \C[x_{i j}, E_{i j}\mid i\in I_0, j\in I_1]\otimes \overline{K}^0_\lambda.
\]
Let us introduce the following operators on $\C[x_{i j}, E_{i j}\mid i\in I_0, j\in I_1]$:
\[
\begin{aligned}
&e =\sum_{i\in I_0, j\in I_1}  x_{i j} E_{i j}, \quad f=\sum_{i\in I_0, j\in I_1} \frac{\partial}{\partial x_{i j}} \frac{\partial}{\partial E_{i j}}, \\
&h_1=m n+\sum_{i\in I_0, j\in I_1} x_{i j}\frac{\partial}{\partial x_{i j}},
\quad h_2=\sum_{i\in I_0, j\in I_1} E_{i j} \frac{\partial}{\partial E_{i j}}.
\end{aligned}
\]
They form the Lie superalgebra $\gl_{1|1}$; note in particular that
\begin{eqnarray}\label{eq:gl11}
{[e, f]} =e f + f e= h_1 - h_2.
\end{eqnarray}
The superspace $\C[x_{i j}, E_{i j}\mid i\in I_0, j\in I_1]$ is a weight module for this $\gl_{1|1}$ with respect to the Cartan subalgebra spanned by $h_1, h_2$. We have the following weight space decomposition:
\begin{eqnarray}\label{eq:wt-decomp}
\phantom{XXX}\C[x_{i j}, E_{i j}\mid i\in I_0, j\in I_1] = \bigoplus_{k\in\Z_+,\,  0\le \alpha\le mn}\C[x_{i j}, E_{i j}\mid i\in I_0, j\in I_1]_{(mn +k, \alpha)}.
\end{eqnarray}

To consider the cohomology, we note that
 the differential operator \eqref{eq:delta-usual-para} can be re-wriiten as
\[
\delta= e\otimes \id_{\overline{K}^0_\lambda}.
\]
Thus $\ker\delta = \ker e\otimes\overline{K}^0_\lambda$ and
$H^*(\fu, K) = \frac{\ker e}{\im e} \otimes \overline{K}^0_\lambda$.
For any $p\in(\ker e)_{(mn +k, \alpha)}$, by \eqref{eq:gl11} we have
\[e f (p) =
[e, f] (p) =  (mn +k - \alpha) p.
\]
This shows that $p\in\im e$ unless $mn +k - \alpha=0$, i.e., $k=0$ and $\alpha=m n$. For $p\in (\ker e)_{(mn, mn)}$, if $p\in\im e$, then there exists $q$ such that $p=e(q)$. This requires that $q$ has weight $(mn-1, mn-1)$, which is not a weight of $\C[x_{i j}, E_{i j}\mid i\in I_0, j\in I_1]$ by \eqref{eq:wt-decomp}. This shows that $(\im e)_{(mn, mn)}=0$.  Hence
\[
H^0(\fu, K) \cong L^0_\lambda, \quad H^i(\fu, K)=0 \ \text{ for all $i>0$}.
\]

\subsubsection{The case of the dual Fock space}

We write $y_{i j}=\frac{\partial}{\partial x_{i j}}$. Then the superspace of cochains is
\[
C_{\Pi\fu^*}(\fu, K)= \C[y_{i j}, E_{i j}\mid i\in I_0, \, j\in I_1]\otimes
\overline{K}^0_\lambda.
\]
Define the following operators on $\C[y_{i j}, E_{i j}\mid i\in I_0, \, j\in I_1]$:
\[
\begin{aligned}
&e =\sum_{i\in I_0, j\in I_1} E_{i j} \frac{\partial}{\partial y_{i j}}, \quad f=\sum_{i\in I_0, j\in I_1} y_{i j} \frac{\partial}{\partial E_{i j}}, \\
&h_1=\sum_{i\in I_0, j\in I_1} E_{i j}\frac{\partial}{\partial E_{i j}},
\quad h_2=\sum_{i\in I_0, j\in I_1} y_{i j} \frac{\partial}{\partial y_{i j}},
\end{aligned}
\]
which give rise to a $\gl_{1|1}$-action. We note in particular the relation
\[
{[e, f]}= e f + f e = h_1 + h_2.
\]
The Cartan subalgebra $\C h_1+\C h_2$ is clearly diagonalisable; we have the
weight space decomposition
\[
\C[y_{i j}, E_{i j}\mid i\in I_0, \, j\in I_1]=\bigoplus_{0\le\alpha\le mn, \, k\in\Z_+}\C[y_{i j}, E_{i j}\mid i\in I_0, \, j\in I_1]_{(\alpha, k)}.
\]
Now $\delta=- e\otimes \id_{\overline{K}^0_\lambda}$, thus $H_{\Pi\fu^*}^*(\fu, K)=\frac{\ker e}{\im e}\otimes \overline{K}^0_\lambda$.  Similar considerations as those in the last section lead to
\[
H_{\Pi\fu^*}^0(\fu, K)\cong  \overline{K}^0_\lambda, \quad H_{\Pi\fu^*}^{-i}(\fu, K)=0 \ \  \text{for all $i>0$}.
\]

\subsection{An example of completed cohomology groups}

We consider the completed $\fu$-cohomology in the case $\gl_{1|2}\supset \fp=\gl_1\oplus\gl_2\oplus\fu\supset\fu=\C E_{12}+\C E_{13}$ and $V=\C^{1|2}$. Note that $\fu$ is purely odd. The differential can be written as the negative of $\delta= x \otimes E_{1 2}+ y\otimes E_{1 3}$, where $x=x_{12},\, y=x_{13}$ are both even variables. We take $D=\C y$, and write $\bar{y}=\partial_y$. Denote the completion of $\CF(D)$ (cf. Section \ref{sect:completion}) by
${\widehat{\mathcal D}}$, which is generated by  $\C[x, \bar{y}]$ and formal power series in $x\bar{y}$.
We have the differential complex
\begin{eqnarray*}
\delta: {\widehat{\mathcal D}}\otimes V \longrightarrow {\widehat{\mathcal D}}\otimes V.
\end{eqnarray*}
For any $f= f_1 v_1 + f_2 v_2 + f_3 v_3$ with $f_i\in {\widehat{\mathcal D}}$, we have $\delta(f) = (x f_2 - \frac{\partial}{\partial_{\bar y}}f_3) v_1$.  It is easy to see that
$
\im\delta=\widehat{\mathcal D}v_1.
$
Now $f\in\ker\delta$ if and only if
\begin{eqnarray}\label{eq:PDE}
x f_2 - \frac{\partial}{\partial_{\bar y}}f_3=0.
\end{eqnarray}
We can solve this at each fixed degree, obtaining the following complete set of solutions
\[
\begin{aligned}
&f_2 = x^p \widehat{P}_p'(x\bar{y}),  \quad  f_3 = x^p \widehat{P}_p(x\bar{y}), \quad p\ge 0; \\
&f_2 = \bar{y}^{-p} \widehat{Q}_p'(x\bar{y}) -p  \bar{y}^{-p}\frac{\widehat{Q}_p(x\bar{y})}{x\bar y},  \quad  f_3 = \bar{y}^{-p} \widehat{Q}_p(x\bar{y}), \quad p<0;
\end{aligned}
\]
where $\widehat{P}_p$ is any power series, and $\widehat{Q}_p$ is any power series without a constant term.  We denote by $\widehat{K}$ the $\C$-span of all $f_2 v_2 + f_3 v_3$ with such $f_2$ and $f_3$.
Then
\[
\begin{aligned}
\ker\delta={\widehat{\mathcal D}} v_1+\widehat{K}, \quad
\im\delta={\widehat{\mathcal D}} v_1, \quad
\HH^*_D(\fu, V) = \widehat{K}.
\end{aligned}
\]

\begin{remark}
In order for the mixed $\fu$-cohomology groups to be $\fl$-modules, we need to take $D$ to be an $\fl$-submodule of $\Pi\fg^*$.
\end{remark}

\section{Concluding remarks}

As already mentioned, the standard cohomology of Lie superalgebras \cite{FL, SchZ} has many important applications, and we also expect it to be the case for the mixed cohomology of Lie superalgebras.
In particular,  we hope that the mixed cohomology will enrich the theory of support varieties of Lie supergroups, and help to remedy failures in the context of Lie supergroups and algebraic supergroups of major classical theorems such as the Bott-Borel-Weil theorem
and Beilinson-Bernstein localisation theorem.

%
%


\end{document}